\documentclass[11pt]{article}
\usepackage[a4paper,margin=1in]{geometry}
\usepackage{amsmath,amssymb,amsthm,mathtools}
\usepackage{microtype}
\usepackage[round,authoryear]{natbib}
\usepackage{authblk}
\makeatletter
\renewcommand\AB@authnote[1]{{\textsuperscript{\normalfont#1}}}
\renewcommand\AB@affilnote[1]{{\textsuperscript{\normalfont#1}}}
\makeatother

\usepackage[hidelinks]{hyperref}

\newtheorem{theorem}{Theorem}[section]
\newtheorem{lemma}[theorem]{Lemma}

\newtheorem{corollary}[theorem]{Corollary}
\theoremstyle{remark}

\newcommand{\1}{\mathbf 1}

\newcommand{\diag}{\operatorname{diag}}

\newcommand{\Span}{\operatorname{span}}

\newcommand{\ind}{\mathbb{I}}

\newcommand{\ResA}{\mathcal A}
\newcommand{\ResB}{\mathcal B}
\newcommand{\Srect}{\mathcal S}
\allowdisplaybreaks[1]
\title{Sharp Normalized Covariance Bounds and\\
Constant-Stretch Correlated Sampling on the Hypersimplex}
\author[1]{Tommaso Cesari}
\author[2]{Roberto Colomboni}

\affil[1]{School of Electrical Engineering and Computer Science, University of Ottawa, Ottawa, Canada}
\affil[2]{School of Mathematics, University of Bristol, Bristol, United Kingdom}

\affil[ ]{\footnotesize
\texttt{tcesari@uottawa.ca},
\texttt{roberto.colomboni@bristol.ac.uk}
}
\date{}

\hypersetup{
 pdftitle={Sharp Normalized Covariance Bounds and Constant-Stretch Correlated Sampling on the Hypersimplex},
 pdfauthor={Tommaso Cesari, Roberto Colomboni}
}

\begin{document}
\maketitle

\begin{abstract}
We establish the normalized covariance bound conjectured by \citet[Conjecture~3]{anarihaqima2026} for fixed-rank
external-field measures.
Let $d\ge2$ be an integer and let $m\in\{1,\ldots,d-1\}$.
Let $w\in(0,+\infty)^d$, and let $\mathsf S$ be an $m$-element
random subset of $[d]\coloneqq \{1,\ldots,d\}$ distributed according to the
rank-$m$ external-field measure with weights $w$, i.e.,
\[
 \mathbb P(\mathsf S=S)=\frac{\prod_{i\in S}w_i}{e_m(w)},
 \qquad S\subseteq[d],\quad |S|=m,
\]
where
\[
 e_m(w)\coloneqq \sum_{\substack{T\subseteq[d]\\ |T|=m}}\prod_{\ell\in T}w_\ell
\]
is the $m$th elementary symmetric polynomial in $w_1,\ldots,w_d$.
Let $X\coloneqq (X_1,\ldots,X_d)^{\top}$ be its indicator vector, i.e.,
\[
 X_i=\ind\{i\in\mathsf S\},\qquad i\in\{1,\ldots,d\}.
\]
Let $\Sigma\coloneqq \operatorname{Cov}(X)$, put $v_i\coloneqq \Sigma_{ii}$ for each
$i\in[d]$, and define
\[
 v\coloneqq (v_1,\ldots,v_d)^{\top},\qquad
 D\coloneqq \operatorname{diag}(v),\qquad V\coloneqq \sum_{i=1}^d v_i.
\]
Our main result is
\[
 \Sigma\succeq D-\frac{vv^{\top}}{V}.
\]
This improves the coefficient $1/2$ established in the first version of our work \citep[Corollary~1.3]{cesaricolomboni2026} to the optimal universal coefficient $1$.
As a first corollary, we improve the coefficient in the pseudoinverse bound of \citet[Lemma~2]{bacchiocchicesaricolomboni2026} from $2$ to the optimal value $1$. Specializing this bound to coordinate differences gives an alternative proof of our effective-resistance theorem from the first version \citep[Theorem~1.1]{cesaricolomboni2026}.
As a further corollary, the framework of \citet[Theorem~1 and Corollary~2]{anarihaqima2026} yields unconditional correlated-sampling guarantees with stretch $6$ on the hypersimplex and $12$ on its at-most variant.
Unconditional constant-stretch guarantees were first established in the first version of our work \citep[Corollaries~1.4 and~1.5]{cesaricolomboni2026}, with constants $16$ and $32$, which we improve here to $6$ and $12$.
These improved constants strengthen the positive resolution, established in the first version of our work, of the constant-stretch question posed by \citet[Theorem~2 and Section~5]{naoretal2026}.
\end{abstract}

\section{Introduction}
\label{sec:introduction}

\subsection{Fixed-rank external fields and covariance}

For every positive integer $r$, write $[r]\coloneqq\{1,\ldots,r\}$ and let $\1_r\in\mathbb R^r$ denote the all-ones vector.
For $S\subseteq[r]$, write $\ind_S\coloneqq(\ind\{i\in S\})_{i=1}^r$ for its indicator vector, where $\ind\{A\}$ denotes the indicator of a proposition $A$.

Let $d\ge2$ be an integer and let $m\in [d-1]$.
Fix $w=(w_1,\ldots,w_d)\in(0,+\infty)^d$, and let $\mathsf S$ be an $m$-element random subset of $[d]$ distributed according to the rank-$m$ external-field measure with weights $w$, i.e.,
\begin{equation}
    \mathbb P(\mathsf S=S)
=
    \frac{\prod_{i\in S}w_i}{e_m(w)},
\qquad
    S\subseteq[d],
    \quad
    |S|=m.
\label{eq:intro-law}
\end{equation}
Here
\[
    e_m(w)
\coloneqq
    \sum_{\substack{T\subseteq[d]\\ |T|=m}}\prod_{\ell\in T}w_\ell
\]
is the $m$th elementary symmetric polynomial in $w_1,\ldots,w_d$.
Set $X\coloneqq\ind_{\mathsf S}\in\{0,1\}^d$.
The law \eqref{eq:intro-law} is the conditional Bernoulli law, also known as rejective sampling \citep{Hajek1964}.
More explicitly, if $B_1,\ldots,B_d$ are independent Bernoulli variables with $\mathbb P(B_i=1)=w_i/(1+w_i)$, then $X$ has the law of $(B_1,\ldots,B_d)^\top$ conditional on $\sum_iB_i=m$.
Equivalently, it is the unique maximum-entropy law among all $m$-subset distributions with the same inclusion marginals \citep{ChenDempsterLiu1994}.

Let
\[
    \mu
\coloneqq
    \mathbb{E}[X],
\qquad
    \Sigma
\coloneqq
    \operatorname{Cov}(X),
\qquad
    v_i
\coloneqq
    \Sigma_{ii}
=
    \mu_i(1-\mu_i).
\]
Since all weights are positive and $1\le m\le d-1$, each coordinate has
positive probability of being included and of being excluded.
Thus $0<\mu_i<1$ and $v_i>0$ for every $i\in[d]$.

Write
\[
    v
=
    (v_1,\ldots,v_d)^\top,
\qquad
    D
=
    \diag(v_1,\ldots,v_d),
\qquad
    V
=
    \1_d^\top v.
\]
For symmetric matrices $A$ and $B$, write $A\succeq B$ if $A-B$ is positive semidefinite.
\subsection{A sharp normalized covariance theorem}
\label{sec:intro-covariance}

Our main result is the following sharp normalized covariance bound.

\begin{theorem}[Sharp normalized covariance bound]
\label{thm:main}
For every fixed-rank external-field law \eqref{eq:intro-law} with positive weights,
\begin{equation}
    \Sigma
\succeq
    D-\frac{vv^\top}{V}.
\label{eq:main}
\end{equation}
Equivalently, every nonzero eigenvalue of $D^{-1/2}\Sigma D^{-1/2}$ is at least $1$.
The constant $1$ is optimal uniformly over the stated family of laws.
\end{theorem}

Anari, Haqi, and Ma conjectured this normalized spectral statement \citep[Conjecture~3]{anarihaqima2026}.
Their first arXiv preprint stated the assertion as a theorem; the revised version explains that the proposed Hodge-Riemann argument does \emph{not} establish the normalized eigenvalue bound, and states the correlated-rounding result conditionally on this conjecture \citep[Section~1, \emph{Status of the main claim}]{anarihaqima2026}.
In the first version of our work, we established the factor-two relaxation
\[
    \Sigma
\succeq
    \frac12\left(D-\frac{vv^\top}{V}\right)
\]
as a consequence of our effective-resistance theorem \citep[Corollary~1.3]{cesaricolomboni2026}.
Theorem~\ref{thm:main} replaces the coefficient $1/2$ by the optimal value $1$ and proves their conjecture.
Its proof is independent of both the earlier effective-resistance theorem and the proposed Hodge-Riemann argument.

The fixed-rank identity $\1_d^\top X=m$ gives $\Sigma\1_d=0$, so every row of $\Sigma$ sums to zero.
Moreover, the calculation leading to \eqref{eq:covariance-laplacian} proves that $\Sigma_{ij}<0$ whenever $i\ne j$.
Define
\[
    c_{ij}
\coloneqq
    -\Sigma_{ij}
>
    0
\qquad
    (i\ne j).
\]
The zero-row-sum identity gives
\[
    \sum_{j\ne i}c_{ij}
=
    v_i.
\]
Recall that the weighted Laplacian $L$ of a graph with edge weights $c_{ij}$ is defined by $L_{ij}=-c_{ij}$ for $i\ne j$ and $L_{ii}=\sum_{j\ne i}c_{ij}$.
Hence $\Sigma=L$ is the Laplacian of the complete graph with edge weights $c_{ij}$, and $v_i$ is the weighted degree of vertex $i$.
Since all edge weights are positive, the graph is connected.
Moreover, for every $x\in\mathbb R^d$,
\[
    x^\top\Sigma x
=
    \sum_{i<j}c_{ij}(x_i-x_j)^2.
\]
We already know that $\Sigma\1_d=0$.
Conversely, if $x\in\ker\Sigma$, then the displayed sum is zero, so $x_i=x_j$ along every edge, and connectivity implies that $x$ is constant.
Hence, $\ker\Sigma=\Span\{\1_d\}$.

Since $D^{-1/2}\Sigma D^{-1/2}$ is a normalized graph Laplacian, its eigenvalues are also at most $2$, and hence
\begin{equation}
    \operatorname{spec}\left(D^{-1/2}\Sigma D^{-1/2}\right)
\subseteq
    \{0\}\cup[1,2].
\label{eq:normalized-spectrum-main}
\end{equation}
Indeed, the Laplacian identity above gives
\[
    x^\top\Sigma x
=
    \sum_{i<j}c_{ij}(x_i-x_j)^2
\le
    2\sum_i v_i x_i^2.
\]

\subsection{Pseudoinverse and effective resistance}
\label{sec:intro-matrix-corollaries}

The Moore-Penrose pseudoinverse $\Sigma^\dagger$ inverts $\Sigma$ on
$(\ker\Sigma)^\perp$ and vanishes on $\ker\Sigma$.
The stronger covariance bound improves the pseudoinverse estimate of
\citet[Lemma~2]{bacchiocchicesaricolomboni2026}, replacing their coefficient
$2$ by the optimal coefficient $1$.

\begin{corollary}[Pseudoinverse bound]\label{cor:inverse}
Under the hypotheses of Theorem~\ref{thm:main},
$\ker\Sigma=\operatorname{span}\{\1_d\}$, and for every
$y\in\mathbb R^d$ with $\1_d^{\top}y=0$,
\begin{equation}\label{eq:inverse-corollary}
 y^{\top}\Sigma^\dagger y\le\sum_{i=1}^d\frac{y_i^2}{v_i},
\end{equation}
where $\Sigma^\dagger$ is the Moore-Penrose pseudoinverse.
The coefficient $1$ is optimal uniformly over this family of laws.
\end{corollary}

Let $\mathbf e_1,\ldots,\mathbf e_d$ denote the standard basis of $\mathbb R^d$.
The effective resistance between vertices $i$ and $j$ in the graph with
Laplacian $\Sigma$ is
\[
 R_{ij}\coloneqq(\mathbf e_i-\mathbf e_j)^\top\Sigma^\dagger(\mathbf e_i-\mathbf e_j).
\]
Taking $y=\mathbf e_i-\mathbf e_j$ in Corollary~\ref{cor:inverse} gives
an alternative proof of our effective-resistance theorem from the first
version of this work \citep[Theorem~1.1]{cesaricolomboni2026}.

\begin{corollary}[Effective resistance]\label{cor:resistance}
Under the hypotheses of Theorem~\ref{thm:main}, for every $i\ne j$,
\begin{equation}\label{eq:resistance-corollary}
 (\mathbf e_i-\mathbf e_j)^{\top}\Sigma^\dagger(\mathbf e_i-\mathbf e_j)
 \le\frac1{v_i}+\frac1{v_j},
\end{equation}
where $\mathbf e_i$ are the standard basis vectors and $\Sigma^\dagger$
is the Moore-Penrose pseudoinverse.
\end{corollary}

The electrical interpretation, graph examples, and asymptotic sharpness of
this resistance bound are discussed in Section~\ref{sec:resistance-details}.

All references to the first version of this paper refer to \href{https://arxiv.org/abs/2607.13990v1}{arXiv:2607.13990v1}, submitted on 15~July 2026 under the title \emph{Effective Resistance in Fixed-Rank External-Field Measures and Constant-Stretch Correlated Sampling on the Hypersimplex}.
All numbered references to Anari, Haqi, and Ma and to Bacchiocchi, Cesari, and Colomboni use version~2 of their respective papers.
\subsection{Application: correlated sampling on the hypersimplex}
\label{sec:intro-sampling}

For any two integers $n\ge1$ and $k \in \{0, \dots, n \}$, introduce the two pieces of notation
\[
    \Delta^{=}_{n,k}
\coloneqq
    \left\{x\in[0,1]^n:\sum_{i=1}^n x_i=k\right\}
\qquad
    \text{and}
\qquad
    \Delta^{\le}_{n,k}
\coloneqq
    \left\{x\in[0,1]^n:\sum_{i=1}^n x_i\le k\right\}.
\]
The first set is the hypersimplex, and the second is its ``at most $k$'' version.
Also write
\[
    \mathcal S_{n,k}
\coloneqq
    \{S\subseteq[n]:|S|=k\}.
\]
A correlated sampler assigns to each input $x\in\Delta^{=}_{n,k}$ a random set $A(x)\in\mathcal S_{n,k}$.
The sets $A(x)$ for different inputs are generated using the same random seed, and each coordinate is included with probability given by the corresponding entry of $x$:
\[
    \mathbb P(i\in A(x))
=
    x_i
\qquad
    \text{for every }x\in\Delta^{=}_{n,k}
    \text{ and every }i\in[n].
\]
Its \emph{stretch} is defined as the smallest constant $\alpha$ such that
\[
    \mathbb E[|A(x)\triangle A(y)|]
\le 
    \alpha\|x-y\|_1,
\]
where $\triangle$ is the symmetric difference operator, the two outputs $A(x)$ and $A(y)$ are generated using the same realization of the sampler's random seed, and the expectation is taken over this shared randomness.
For inputs $x\in\Delta^{\le}_{n,k}$, the formulation of
\citet[Definition~1]{naoretal2026} additionally requires the output cardinality to satisfy
\[
    |A(x)|
\in
    \left\{\left\lfloor\sum_{i=1}^n x_i\right\rfloor,\left\lceil\sum_{i=1}^n x_i\right\rceil\right\}
\]
for every realization of the common random seed.
After an $O(\log n)$ result of \citet[Theorem~1]{ChenEtAl2017},
\citet[Theorem~2 and Section~5]{naoretal2026} obtained $O(\log k)$ stretch, independent of the
ambient dimension, and asked whether the logarithmic dependence is
inherent.

\citet{anarihaqima2026} proposed the following canonical scheme.
For $x\in\Delta^{=}_{n,k}$, take the maximum-entropy law $\mu_x$ on $k$-subsets with inclusion marginals $x$.  Draw a uniformly random ordering $(i_1,\ldots,i_n)$ of $[n]$ and, independently, draw mutually independent random variables $U_1,\ldots,U_n$, each uniformly distributed on $[0,1]$.
Process the coordinates in this order.  When coordinate $i_t$ is processed, compute its inclusion probability under $\mu_x$, conditional on all preceding inclusion and exclusion decisions. If this probability lies strictly between zero and one, include $i_t$ if and only if $U_{i_t}$ does not exceed that probability.
Coordinates with conditional inclusion probability zero or one are excluded or included deterministically.
The same ordering and the same random variables are reused for every input $x$.
Their analysis provides a constant-stretch guarantee conditional on the validity of two uniform estimates for the covariance Laplacians: a lower bound on the nonzero eigenvalues of the normalized Laplacian and a weighted edge-variation bound for the solutions of the corresponding Laplacian systems with right-hand side $\mathbf e_a-\mathbf e_b$.
In their work, these estimates were obtained conditionally on the normalized covariance conjecture discussed above \citep[Conjecture~3 and Theorem~11]{anarihaqima2026}.
In the first version of this paper, we instead established the two estimates using the half-normalized covariance bound and the effective-resistance theorem \citep[Corollary~1.3 and Proposition~3.2]{cesaricolomboni2026}.
Combining those estimates with their analysis gave unconditional stretch $16$ on the hypersimplex and $32$ on its at-most variant \citep[Corollaries~1.4 and~1.5]{cesaricolomboni2026}.
Theorem~\ref{thm:main} now proves the conjecture itself.
Applying the conditional results of \citet[Theorem~1 and Corollary~2]{anarihaqima2026} therefore improves these already unconditional guarantees to $6$ and $12$.
For $N\ge1$, let
\[
    \Omega_N
\coloneqq
    \mathfrak S_N\times[0,1]^N,
\]
where $\mathfrak S_N$ denotes the set of permutations of $[N]$.
We equip $\Omega_N$ with the product probability measure under which the permutation is uniform on $\mathfrak S_N$ and the $N$ coordinates in $[0,1]^N$ are independent and uniformly distributed.

For every $x\in\Delta^{=}_{n,k}$, we denote by $\mu_x$ the maximum-entropy law on $k$-element subsets with marginals $x$.

\begin{corollary}[Exact hypersimplex]
\label{cor:rounding}
For every integer $n\ge1$ and every $k\in\{0,\ldots,n\}$, the Anari-Haqi-Ma common-seed sampler has a jointly Borel measurable version
\[
    A_{n,k}:\Delta^{=}_{n,k}\times\Omega_n
\to
    \mathcal S_{n,k}
\]
such that
\[
    A_{n,k}(x,\cdot)\sim\mu_x
\qquad
    \text{for every }x\in\Delta^{=}_{n,k},
\]
and
\begin{equation}
    \mathbb E_\omega\left[\left|A_{n,k}(x,\omega)\triangle A_{n,k}(y,\omega)\right|\right]
\le
    6\|x-y\|_1
\qquad
    \text{for every }x,y\in\Delta^{=}_{n,k}.
\label{eq:exact-stretch-main}
\end{equation}
\end{corollary}

\begin{corollary}[At most $k$ elements]
\label{cor:rounding-atmost}
For every integer $n\ge1$ and every $k\in\{0,\ldots,n\}$, there
exists a jointly Borel measurable map
\[
    B_{n,k}:\Delta^{\le}_{n,k}\times\Omega_{n+k}
\to
    \{S\subseteq[n]:|S|\le k\}
\]
such that
\[
    \left|B_{n,k}(x,\omega)\right|
\in
    \left\{\left\lfloor\sum_{i=1}^n x_i\right\rfloor,\left\lceil\sum_{i=1}^n x_i\right\rceil\right\}
\]
for every $x\in\Delta^{\le}_{n,k}$ and every $\omega\in\Omega_{n+k}$.
Moreover,
\[
    \mathbb P_\omega\bigl(i\in B_{n,k}(x,\omega)\bigr)
=
    x_i
\qquad
    \text{for every }x\in\Delta^{\le}_{n,k}
    \text{ and every }i\in[n],
\]
and
\begin{equation}
    \mathbb E_\omega\left[\left|B_{n,k}(x,\omega)\triangle B_{n,k}(y,\omega)\right|\right]
\le
    12\|x-y\|_1
\qquad
    \text{for every }x,y\in\Delta^{\le}_{n,k}.
\label{eq:at-most-stretch-main}
\end{equation}
\end{corollary}

Thus, the constant-versus-$O(\log k)$ question of \citet[Section~5]{naoretal2026} has a positive answer, as already shown in the first version of this paper.

\subsection{Manuscript organization}

The rest of this manuscript is organized as follows.
Section~\ref{sec:symmetric-identities} develops the symmetric-function identities used in the proof.
Section~\ref{sec:scalar-positivity} proves the scalar positivity statement by expanding it in the Schur basis.
Section~\ref{sec:deletion-inertia} uses this statement to determine the inertia of the deletion matrix.
Section~\ref{sec:covariance-proof} then proves Theorem~\ref{thm:main} and its sharpness.
Section~\ref{sec:corollaries} proves the pseudoinverse, effective-resistance, and correlated-sampling corollaries.

\section{Symmetric-function identities}\label{sec:symmetric-identities}
For an alphabet $a=(a_1,\ldots,a_N)$, let $a_{-i}$ and
$a_{-i\ell}$ denote deletion of the indicated letters.
Writing $(x,A)$ or $(x,y,A)$ means adjoining the displayed letters
to the alphabet $A$.
We use
\[
 \sum_{q\ge0}e_q(a)t^q=\prod_{i=1}^N(1+a_i t),\qquad
 \sum_{q\ge0}h_q(a)t^q=\prod_{i=1}^N(1-a_i t)^{-1}.
\]
In particular, $e_0=h_0=1$, and $e_q=h_q=0$ for $q<0$.
For finite alphabets, $e_q(a)=0$ also when $q>N$.
We put $p_q(a)=\sum_i a_i^q$ for $q\ge1$.

A partition is a finite weakly decreasing sequence of positive
integers; it may be padded with zeros when used in a determinant.
Its length is its number of positive parts. Repeated parts are written
as powers, so $(q^h)$ denotes the rectangle with $h$ rows of
length $q$. Zero multiplicities and trailing zero parts are omitted.
The empty partition is denoted by $\varnothing$, and its first part is interpreted as zero.
An empty product and an empty determinant are both one.
The empty alphabet is allowed in intermediate deletion identities.

\subsection{Schur polynomials, determinants, and Pieri's rule}
The Schur polynomial has the tableau expansion
\begin{equation}\label{eq:schur-tableaux}
 s_\lambda(a)=\sum_T\prod_{u\in\lambda}a_{T(u)},
\end{equation}
where $T$ runs over fillings of the Young diagram of $\lambda$
by $1,\ldots,N$ that are weakly increasing along rows and strictly
increasing down columns; this is \citet[Chapter~I, (5.12), p.~73]{macdonald1995} with empty inner shape.
We set $s_\varnothing=1$.
For positive letters all summands are nonnegative.
If $\lambda$ has at most $N$ rows, filling row $i$ entirely
with the entry $i$ gives an admissible tableau with positive weight.
Consequently $s_\lambda(a)>0$ in this case.
If $\lambda$ has more than $N$ rows, its first column cannot be
filled strictly increasingly, and $s_\lambda(a)=0$.

We use the Jacobi-Trudi identity and its dual
\citep[Chapter~I, (3.4) and (3.5), p.~41]{macdonald1995}:
\begin{align}
 s_\lambda(a)
 &=\det\bigl(h_{\lambda_i-i+j}(a)\bigr)_{i,j=1}^{L},
 \qquad L\ge\operatorname{length}(\lambda),\label{eq:jt}\\
 s_\lambda(a)
 &=\det\bigl(e_{\lambda'_i-i+j}(a)\bigr)_{i,j=1}^{Q},
 \qquad Q\ge\lambda_1,\label{eq:dual-jt}
\end{align}
where $\lambda'$ is the conjugate partition.
In particular,
\[
 s_{(q^h)}(a)=\det(e_{h-i+j}(a))_{i,j=1}^q.
\]
The vertical form of Pieri's rule, \citep[Chapter~I, (5.17), p.~73]{macdonald1995}, is
\begin{equation}\label{eq:pieri}
 e_t(a)s_\mu(a)
 =\sum_{\lambda:\,\lambda/\mu\text{ is a vertical }t\text{-strip}}
 s_\lambda(a).
\end{equation}
Here a vertical $t$-strip consists of $t$ new boxes with at most
one new box in each row; the resulting diagram must be a partition.
Each allowed resulting partition occurs with coefficient one.

For distinct letters, the bialternant formula
\citet[Chapter~I, (3.1), p.~40]{macdonald1995} gives
\begin{equation}\label{eq:bialternant}
 s_\lambda(a)
 =\frac{\det(a_i^{\lambda_j+N-j})_{i,j=1}^N}
        {\det(a_i^{N-j})_{i,j=1}^N},
 \qquad\operatorname{length}(\lambda)\le N.
\end{equation}
The quotient is the polynomial $s_\lambda$; identities obtained
from it for distinct letters therefore extend to repeated letters
by polynomial identity or continuity.

Schur functions form a basis of the graded ring of symmetric functions,
\citep[Chapter~I, (3.3), p.~41]{macdonald1995}.
Thus $[s_\lambda]G$ denotes the coefficient of $s_\lambda$
in the stable Schur expansion of $G$. The function $G$ is called
Schur-positive when all these coefficients are nonnegative.
Specializing to a finite alphabet sets to zero the Schur polynomials
whose diagrams have more rows than there are letters.

We shall write
\begin{equation}\label{eq:K-S-definitions}
 K_j=e_j^2-e_{j-1}e_{j+1}=s_{(2^j)}\quad(j\ge1),\qquad
 K_0=1,\qquad \Srect_r=s_{(3^{r-1})}\quad(r\ge1).
\end{equation}
The equality defining $K_j$ follows from the two-by-two
determinant in \eqref{eq:dual-jt}.
All these functions are evaluated on the indicated alphabet;
when no alphabet is shown, they are evaluated on the current full
alphabet.

\subsection{Deleting a letter and inverting the alphabet}
The next two identities follow from the determinant formulas above.

\begin{lemma}[Deletion from a rectangle]\label{lem:rect-deletion}
For integers $q\ge1$ and $k\ge0$,
\begin{equation}\label{eq:rect-deletion}
 s_{(q^k)}[a-x]
 =\sum_{j=0}^{k}(-x)^j s_{(q^{k-j},(q-1)^j)}(a).
\end{equation}
The notation $[a-x]$ means that every complete symmetric function
in \eqref{eq:jt} is replaced by
$h_d[a-x]=h_d(a)-x h_{d-1}(a)$.
When $x=a_i$, the left-hand side is $s_{(q^k)}(a_{-i})$.
\end{lemma}
\begin{proof}
The last assertion follows from the generating function:
\[
 \sum_{d\ge0}\bigl(h_d(a)-a_i h_{d-1}(a)\bigr)t^d
 =(1-a_i t)\prod_{\ell=1}^N(1-a_\ell t)^{-1}
 =\prod_{\ell\ne i}(1-a_\ell t)^{-1}.
\]
For $k=0$, both sides of \eqref{eq:rect-deletion} equal one.
For $k\ge1$, put
\[
 U_i=(h_{q-i+1}(a),\ldots,h_{q-i+k}(a)),\qquad 1\le i\le k+1.
\]
The determinant on the left of \eqref{eq:rect-deletion} has rows
$U_i-xU_{i+1}$, $1\le i\le k$.
Expand it multilinearly in the rows.
If a term chooses $-xU_{i+1}$ in row $i$ and $U_{i+1}$
in row $i+1$, it has two equal rows and vanishes.
Thus a nonvanishing term must choose the shifted row only in
a terminal block of rows, say in precisely the last $j$ rows.
Its coefficient is $(-x)^j$; its rows are those of the
Jacobi-Trudi determinant for
$(q^{k-j},(q-1)^j)$.
Summing over $j=0,\ldots,k$ proves the formula.
\end{proof}

\begin{lemma}[Complementation in a rectangle]\label{lem:complement}
Let $a$ be a positive $N$-letter alphabet, put
$\Pi=\prod_i a_i$, and suppose
$q$ is a nonnegative integer with $q\ge\lambda_1$, and $\operatorname{length}(\lambda)\le N$.
Pad $\lambda$ to length $N$ and set
$\lambda^c=(q-\lambda_N,\ldots,q-\lambda_1)$.
Then
\begin{equation}\label{eq:complement}
 \Pi^q s_\lambda(a^{-1})=s_{\lambda^c}(a),
 \qquad a^{-1}=(a_1^{-1},\ldots,a_N^{-1}).
\end{equation}
In particular, for $0\le h\le N$,
\begin{equation}\label{eq:rect-complement}
 s_{(q^h)}(a)=\Pi^q s_{(q^{N-h})}(a^{-1}),\qquad
 e_h(a)=\Pi e_{N-h}(a^{-1}).
\end{equation}
\end{lemma}
\begin{proof}
First suppose the letters are pairwise distinct.
Apply \eqref{eq:bialternant} to $a^{-1}$.
Multiplying row $i$ in the numerator by $a_i^{q+N-1}$
and row $i$ in the denominator by $a_i^{N-1}$ gives
\[
 \Pi^q s_\lambda(a^{-1})
 =\frac{\det(a_i^{q+j-1-\lambda_j})_{i,j=1}^N}
        {\det(a_i^{j-1})_{i,j=1}^N}.
\]
Reverse the columns in both determinants. Their reversal signs
cancel, and the numerator exponents become
$q-\lambda_{N+1-j}+N-j=\lambda^c_j+N-j$.
Equation~\eqref{eq:bialternant} now gives \eqref{eq:complement}.
Continuity extends the identity to all positive alphabets.
Taking $\lambda$ rectangular proves the first formula in
\eqref{eq:rect-complement}; taking $q=1$ proves the second.
\end{proof}

\subsection{Determinant condensation}
For a square matrix $M$, let $M^I_J$ be the matrix obtained
by deleting the rows indexed by $I$ and columns indexed by $J$.
The Desnanot-Jacobi identity, Proposition~10, equation~(2.16), of \citet{krattenthaler1999}, states that,
for an $r$-by-$r$ matrix with $r\ge2$,
\begin{equation}\label{eq:desnanot-jacobi}
 \det M\,\det M^{\{1,r\}}_{\{1,r\}}
 =\det M^{\{1\}}_{\{1\}}\,\det M^{\{r\}}_{\{r\}}
  -\det M^{\{1\}}_{\{r\}}\,\det M^{\{r\}}_{\{1\}}.
\end{equation}
The determinant of the empty matrix is one, so the identity also
includes $r=2$. The remaining determinant identities follow by expansion or condensation.

\section{Scalar positivity}\label{sec:scalar-positivity}
For $k\ge0$, define
\begin{equation}\label{eq:scalars}
\begin{split}
 F_k(a)&=\sum_i a_i^2s_{(4^k)}(a_{-i}),\\
 \ResB_k(a)&=s_{(3^k)}(a)s_{(2^{k+1})}(a)-e_k(a)F_k(a),\\
 \ResA_k(a)&=s_{(2^k)}(a)s_{(3^{k+1})}(a)-e_{k+1}(a)F_k(a).
\end{split}
\end{equation}
The following Schur expansion holds on every finite nonnegative
alphabet. The coefficients are independent of the alphabet;
partitions with too many rows contribute zero.

\begin{lemma}\label{lem:scalar}
For every finite alphabet $a$ of nonnegative real numbers,
\[
 \ResB_k(a)=\sum_{\lambda\vdash 5k+2}c_\lambda s_\lambda(a),
\]
where all $c_\lambda$ are nonnegative integers and
$c_{(5^k,1,1)}=1$.
More precisely, for a partition
$\lambda=(5^\alpha,4^\beta,3^\gamma,2^\delta,1^\epsilon)$
of degree $5k+2$,
\begin{equation}\label{eq:coeff}
 c_\lambda=
 \begin{cases}
 \beta+1,&\alpha+\beta+\gamma=k,\ \delta=0,\\
 \epsilon+1,&\alpha+\beta+\gamma=k+1,\ \gamma=0,\\
 0,&\text{otherwise}.
 \end{cases}
\end{equation}
Partitions not of this form have coefficient zero.
In particular, $\ResB_k(a)\ge s_{(5^k,1,1)}(a)\ge0$.
\end{lemma}
\begin{proof}
For $k=0$, the definition gives $F_0=p_2=\sum_i a_i^2$.
Since
\[
 s_{(2)}=h_2=\sum_i a_i^2+\sum_{i<j}a_i a_j,
\]
we have $\ResB_0=s_{(2)}-p_2=e_2=s_{(1,1)}$.
This is also exactly \eqref{eq:coeff} in degree two.
Henceforth suppose $k\ge1$.

\medskip
\noindent\emph{Step 1: a signed Schur expansion of the deletion sum.}
We first prove
\begin{equation}\label{eq:load}
 F_k=\sum_{j=0}^{k}(P_j-Q_j)-\sum_{j=0}^{k-1}Z_j,
\end{equation}
where
\[
 P_j=s_{(4^{k-j},3^j,2,1^j)},\qquad
 Q_j=s_{(4^{k-j},3^j,1^{j+2})},\qquad
 Z_j=s_{(4^{k-j-1},3^{j+2},1^j)}.
\]
All three types have degree $4k+2$.
The one-letter deletion formula of Lemma~\ref{lem:rect-deletion},
applied with the deleted letter equal to $a_i$, gives
\[
 s_{(4^k)}(a_{-i})
 =\sum_{j=0}^{k}(-a_i)^j
       s_{(4^{k-j},3^j)}(a).
\]
Multiplying by $a_i^2$, summing over $i$, and writing
$p_t(a)=\sum_i a_i^t$, we obtain
\begin{equation}\label{eq:load-power-sums}
 F_k=\sum_{j=0}^{k}(-1)^j
       p_{j+2}s_{(4^{k-j},3^j)}.
\end{equation}
We expand the products on the right directly by alternants.

Fix $N\ge2k+3$. For an integer vector
$c=(c_1,\ldots,c_N)$ with $c_i+N\ge0$, put
\[
 \mathfrak a(c)=\det(a_s^{c_i+N})_{i,s=1}^N.
\]
For a partition $\mu$ with at most $N$ rows, padded with
zeros, the bialternant formula \eqref{eq:bialternant} reads
\[
 s_\mu(a)=
 \frac{\mathfrak a(\mu_1-1,\ldots,\mu_N-N)}
      {\mathfrak a(-1,\ldots,-N)}.
\]
These quotients may be computed in the field of rational functions
in independent variables; the resulting identities are polynomial
identities and hence also hold when some letters coincide.
Expanding the determinant by permutations proves
\[
 p_t\mathfrak a(c)
   =\sum_{u=1}^{N}\mathfrak a(c+t\,\mathbf e_u).
\]
Here $\mathbf e_u$ is the $u$-th coordinate vector.
Indeed, in each permutation monomial, multiplying by
$\sum_s a_s^t$ increases the exponent of exactly one of its
$N$ factors, and those factors correspond bijectively to the
rows of the determinant. Thus multiplication by $p_t$ raises
one occupied exponent $c_u$ by $t$.
If the new exponent equals another occupied exponent, the
determinant has two equal rows and vanishes. Otherwise, restoring
decreasing order contributes the sign $(-1)^q$, where $q$
is the number of occupied exponents strictly between the source
and its destination. The reordered exponents determine the
resulting Schur function.

For $\mu=(4^{k-j},3^j)$, the occupied exponents are
\[
 \bigl(I\setminus\{b_j\}\bigr)
       \cup\{-k-1,-k-2,\ldots,-N\},\qquad
 I=\{3-k,4-k,\ldots,3\},\qquad b_j=3-k+j.
\]
The first $k-j$ rows have exponents
$4-i$, the next $j$ rows have exponents $3-i$, and all
remaining rows have exponents $-i$. There are consequently
three fixed gaps $-k,-k+1,-k+2$, as well as the moving gap
$b_j$ inside $I$.

Consider first a source $b_\ell\in I\setminus\{b_j\}$, so
$0\le\ell\le k$ and $\ell\ne j$. Its destination in the
$j$-th summand of \eqref{eq:load-power-sums} is
\[
 b_\ell+j+2=5-k+j+\ell.
\]
It cannot equal the gap $b_j$, since that would require
$\ell=-2$. Every other destination in $I$ is occupied.
The source therefore survives exactly when its destination
exceeds three. In that case, pair it with the source $b_j$
in the $\ell$-th summand. That source has the same destination,
and both final sets of exponents are
\[
 \bigl(I\setminus\{b_j,b_\ell\}\bigr)
 \cup\{5-k+j+\ell\}
 \cup\{-k-1,\ldots,-N\}.
\]
The numbers of crossed exponents are respectively
$k-\ell-\ind\{j>\ell\}$ and $k-j-\ind\{\ell>j\}$.
After including the factors $(-1)^j$ and $(-1)^\ell$ in
\eqref{eq:load-power-sums}, the two sign exponents are
\[
 j+k-\ell-\ind\{j>\ell\},\qquad
 \ell+k-j-\ind\{\ell>j\}.
\]
Their difference is
$2(j-\ell)-\ind\{j>\ell\}+\ind\{\ell>j\}$, an odd integer.
The pairing is an involution without fixed points, so every
surviving source in $I$ cancels.

It remains to consider a source in the tail
$\{-k-1,\ldots,-N\}$. Its destination is at most
$-k-1+(k+2)=1$, whereas the largest occupied upper exponent is
$3$ when $j<k$ and $2$ when $j=k$. Thus it cannot jump
above the occupied upper block. A destination in the occupied tail or in
$I\setminus\{b_j\}$ gives a collision. Landing at $b_j$
would require the source
$b_j-(j+2)=1-k$, which is a fixed gap and is not occupied.
Every surviving tail source must therefore land at
$-k+d$, for $d\in\{0,1,2\}$. Its source is necessarily
\[
 q=-k+d-j-2.
\]
The condition $q\le-k-1$ is exactly $j\ge d-1$.
Moreover $q\ge-2k-2>-N$, so all such sources are present
in the chosen finite alternant. Between $q$ and its destination
there are exactly $j+1-d$ occupied tail exponents and no
occupied exponents from $I$. Including the prefactor
$(-1)^j$, the total sign is
\[
 (-1)^j(-1)^{j+1-d}=(-1)^{1-d}.
\]
The inserted exponent is the new $(k+1)$-st exponent; it
therefore produces a row of length $d+1$. Removing the source
from the tail leaves $j+1-d$ occupied exponents from
$-k-1$ down to $q+1$. Each of these has moved one place
later in the decreasing ordering, so each gives a row of length
one. The exponents below $q$ retain their original positions
and give zero parts.
The possibilities are
\[
\begin{array}{c|c|c|c}
 d&\text{range of }j&\text{resulting partition}&\text{signed term}\\ \hline
 0&0\le j\le k&(4^{k-j},3^j,1^{j+2})&-Q_j\\
 1&0\le j\le k&(4^{k-j},3^j,2,1^j)& P_j\\
 2&1\le j\le k&(4^{k-j},3^{j+1},1^{j-1})&-Z_{j-1}.
\end{array}
\]
This accounts for every source and proves \eqref{eq:load} in
every number of variables $N\ge2k+3$. Both sides are symmetric
functions of the fixed degree $4k+2$. The identity is therefore
stable, and setting surplus variables equal to zero proves it
on every finite alphabet.

\medskip
\noindent\emph{Step 2: the contribution of the initial product.}
By dual Jacobi-Trudi \eqref{eq:dual-jt},
\[
 s_{(2^{k+1})}=e_{k+1}^2-e_ke_{k+2}.
\]
Consequently
\begin{equation}\label{eq:B-expanded}
 \ResB_k
 =s_{(3^k)}(e_{k+1}^2-e_ke_{k+2})
      +e_k\sum_{j=0}^{k}Q_j
      -e_k\sum_{j=0}^{k}P_j
      +e_k\sum_{j=0}^{k-1}Z_j.
\end{equation}
We use vertical Pieri \eqref{eq:pieri} to compute each coefficient.
A vertical strip adds at most one box to each row. The first
term involves two strips added to $(3^k)$: its original
$k$ rows have final lengths three, four, or five, and every
new row has final length at most two. In a term of the signed Schur expansion, the old
rows of lengths three or four stay at least three and become
at most five. The distinguished two-box row of $P_j$ may
either remain two or become three; each old one-box row becomes
at most two; each newly created row has length one.
Thus every possible target has the form
\[
 \lambda=(5^\alpha,4^\beta,3^\gamma,2^\delta,1^\epsilon),
 \qquad |\lambda|=5k+2,
\]
and lies in precisely one of the two sectors
$\alpha+\beta+\gamma=k$ or
$\alpha+\beta+\gamma=k+1$.

For the initial product only the first sector is possible.
Define, for integers $b,e\ge0$,
\[
 N_{b,e}(u)=\#\{(x,y)\in\mathbb Z^2:
                  0\le x\le b,\ 0\le y\le e,\ x+y=u\}.
\]
Consider an intermediate partition obtained after a first
vertical strip of size $t$, followed by a second strip
leading to $\lambda$. The $\alpha$ original rows that
finish at five must receive a box in each strip. Likewise,
the $\delta$ new rows that finish at two must receive a box
in each strip. Among the $\beta$ original rows that finish
at four, let $x$ receive their box in the first strip;
among the $\epsilon$ new rows that finish at one, let $y$
appear in the first strip. In either block the rows chosen
in the first strip must be an initial segment, because the
intermediate row lengths must be nonincreasing. Hence the
intermediate partition is determined uniquely by $(x,y)$.
Conversely, every $0\le x\le\beta$,
$0\le y\le\epsilon$ gives a valid pair of strips, provided
their first size is
\[
 t=\alpha+\delta+x+y.
\]
The second size is then determined by total degree.
Thus the coefficient in the initial product is
\begin{equation}\label{eq:initial-count}
 N_{\beta,\epsilon}(k+1-\alpha-\delta)
  -N_{\beta,\epsilon}(k-\alpha-\delta).
\end{equation}
Here the first term counts a first strip of size $k+1$
and a second of size $k+1$; the second counts a first strip
of size $k$ and a second of size $k+2$.

The degree condition and $\alpha+\beta+\gamma=k$ give
\begin{equation}\label{eq:first-sector-degree}
 \epsilon=\beta+2\gamma+2-2\delta,
 \qquad
 2(k+1-\alpha-\delta)=\beta+\epsilon.
\end{equation}
In particular, $\beta$ and $\epsilon$ have the same parity.
The counting function is
\[
 N_{b,e}(u)=
 \max\{0,\min(b,u)-\max(0,u-e)+1\}.
\]
If $b=e$, its values at $u=(b+e)/2=b$ and $u-1$
are $b+1$ and $b$, respectively, including $b=0$.
If $b<e$ and the two have the same parity, then
$e\ge b+2$; both central arguments belong to the plateau
$[b,e]$, on which the count is $b+1$.
The case $e<b$ follows by interchanging $b$ and $e$.
It follows that \eqref{eq:initial-count} is one if
$\beta=\epsilon$ and zero otherwise. By
\eqref{eq:first-sector-degree}, the condition is equivalently
\[
 \delta=\gamma+1,\qquad \epsilon=\beta.
\]

\medskip
\noindent\emph{Step 3: the sector
$\alpha+\beta+\gamma=k$.}
All interval cardinalities that follow count integers; an empty
interval or intersection has cardinality zero. For a fixed source
partition and a fixed target, Pieri contributes either zero or
one. Indeed, the skew diagram is determined by the two partitions,
and its size is automatically $k$, since every source Schur polynomial
has degree $4k+2$.

Start with
$Q_j=s_{(4^{k-j},3^j,1^{j+2})}$.
Of its $k-j$ four-box rows, exactly $\alpha$ become five;
of its $j$ three-box rows, exactly $\gamma$ stay three.
The remaining rows from these two blocks become the
$\beta$ four-box rows. These requirements are possible
exactly when
\[
 \gamma\le j\le\beta+\gamma.
\]
The number of three-box rows
raised to four is $j-\gamma$, which must lie in
$[0,\beta]$. The number of original four-box rows left
unchanged is then $\beta-(j-\gamma)$, and the number raised
to five is $\alpha$, as required by
$\alpha+\beta+\gamma=k$.
Among the $j+2$ one-box rows, exactly $\delta$ become
two. The remaining $j+2-\delta$ rows stay one, and any
additional one-box rows are newly created. The conditions are
therefore
\[
 \delta\le j+2,\qquad j+2-\delta\le\epsilon,
 \quad\text{or equivalently}\quad
 \delta-2\le j\le\delta+\epsilon-2.
\]
Within each equal-row block the raised rows must be its initial
segment, so these conditions are also sufficient and determine
a unique vertical strip. Moreover,
$0\le\gamma\le j\le\beta+\gamma=k-\alpha\le k$,
so the original source range $0\le j\le k$ imposes no
additional restriction. The total positive contribution is
\[
 Q=\#\bigl([\gamma,\beta+\gamma]
                 \cap[\delta-2,\delta+\epsilon-2]\bigr).
\]

For $P_j=s_{(4^{k-j},3^j,2,1^j)}$, the distinguished
two-box row must remain two in this sector. In particular
$\delta\ge1$. The old four-box and three-box rows give
the same first interval as above. Exactly $\delta-1$
of its $j$ one-box rows must become two, leaving
$j-\delta+1$ old one-box rows. The conditions are
\[
 \delta-1\le j,\qquad j-\delta+1\le\epsilon.
\]
They again determine a unique strip, with source bounds
already implied by the first interval. Thus the negative
contribution is
\[
 P=\#\bigl([\gamma,\beta+\gamma]
                 \cap[\delta-1,\delta+\epsilon-1]\bigr)
 \quad(\delta\ge1),
\]
and $P=0$ if $\delta=0$. There is no $Z_j$
contribution in this sector, since every $Z_j$ already has
$k+1$ rows of length at least three.

If $\delta=0$, equation \eqref{eq:first-sector-degree} gives
$\epsilon=\beta+2\gamma+2$. The interval
$[-2,\epsilon-2]$ then contains
$[\gamma,\beta+\gamma]$, so $Q=\beta+1$.
The initial-product coefficient is zero because
$\delta=\gamma+1$ cannot hold. The resulting coefficient
of $s_\lambda$ is therefore $\beta+1$.

Suppose instead that $\delta\ge1$, and put
$u=\delta-\gamma-1$. Shifting the source index by
$\gamma$ and using \eqref{eq:first-sector-degree} gives
\[
 \epsilon=\beta-2u,\qquad
 Q=\#\bigl([0,\beta]\cap[u-1,\beta-u-1]\bigr),\qquad
 P=\#\bigl([0,\beta]\cap[u,\beta-u]\bigr).
\]
The counts are evaluated as follows:
if $u\le-1$, both intersections equal $[0,\beta]$;
if $u=0$, the counts are $\beta$ and $\beta+1$;
if $u\ge1$, the condition $\epsilon\ge0$ implies
$\beta\ge2u$, and both intervals lie inside
$[0,\beta]$ with cardinality $\beta-2u+1$.
Consequently $Q-P=-1$ only when $u=0$, and is zero
otherwise. The initial product contributes one exactly when
$\delta=\gamma+1$, which is the same condition $u=0$.
The full coefficient in this sector therefore vanishes whenever
$\delta\ge1$. This includes $\beta=0$ and
$\epsilon=0$.

\medskip
\noindent\emph{Step 4: the sector
$\alpha+\beta+\gamma=k+1$.}
The degree condition now becomes
\begin{equation}\label{eq:second-sector-degree}
 \epsilon=\beta+2\gamma-3-2\delta.
\end{equation}
There is no initial-product contribution in this sector, nor
any $Q_j$ contribution, because neither can produce an
additional row of length three.

For $Z_j=s_{(4^{k-j-1},3^{j+2},1^j)}$, exactly
$\alpha$ old four-box rows become five and exactly
$\gamma$ old three-box rows stay three. The number
of old three-box rows raised to four is $j+2-\gamma$,
and it must lie in $[0,\beta]$. This gives
\[
 \gamma-2\le j\le\beta+\gamma-2.
\]
Among the $j$ one-box rows, exactly $\delta$ become
two, and the $j-\delta$ that remain one must fit into
the target's $\epsilon$ one-box rows. Thus
$\delta\le j\le\delta+\epsilon$.
As before these conditions are necessary and sufficient
and specify a unique strip. They imply the source bounds:
\[
 j\ge\delta\ge0,\qquad
 j\le\beta+\gamma-2=k-\alpha-1\le k-1.
\]
Accordingly, the positive contribution is
\[
 Z=\#\bigl([\gamma-2,\beta+\gamma-2]
                    \cap[\delta,\delta+\epsilon]\bigr).
\]

For $P_j$, the distinguished two-box row must now become
three. Hence $\gamma\ge1$, and exactly $\gamma-1$
of the $j$ original three-box rows remain three. The number
raised to four is $j-\gamma+1$, which must lie in
$[0,\beta]$. This gives
\[
 \gamma-1\le j\le\beta+\gamma-1.
\]
The one-box rows give exactly the same restrictions
$\delta\le j\le\delta+\epsilon$ as for $Z_j$.
The source range is again automatic, since
\[
 j\ge\delta\ge0,\qquad
 j\le\beta+\gamma-1=k-\alpha\le k.
\]
Thus the negative contribution is
\[
 P=\#\bigl([\gamma-1,\beta+\gamma-1]
                    \cap[\delta,\delta+\epsilon]\bigr)
 \quad(\gamma\ge1),
\]
and $P=0$ if $\gamma=0$.

When $\gamma\ge1$, set $u=\delta-\gamma+1$.
For $Z$, substitute $j=\gamma-2+x$; for $P$,
substitute $j=\gamma-1+x$. Equation
\eqref{eq:second-sector-degree} gives
\[
 Z=\#\bigl([0,\beta]\cap[u+1,\beta-u]\bigr),\qquad
 P=\#\bigl([0,\beta]\cap[u,\beta-u-1]\bigr).
\]
The map $x\mapsto\beta-x$ is a bijection between the two
integer sets: it preserves $[0,\beta]$ and takes
$[u+1,\beta-u]$ onto $[u,\beta-u-1]$.
Hence $Z=P$, also when the sets are empty.
The only potentially negative unshifted source is $j=-1$
in the $Z$ substitution when $\gamma=1$; it does not
occur, because then $u=\delta\ge0$ and the first set
requires $x\ge u+1\ge1$.

Finally, if $\gamma=0$, the $P$ contribution is absent,
and \eqref{eq:second-sector-degree} reads
$\epsilon=\beta-3-2\delta\ge0$. We have
\[
 \delta\ge0>-2,\qquad
 \delta+\epsilon=\beta-3-\delta\le\beta-3<\beta-2.
\]
Thus $[\delta,\delta+\epsilon]$ is contained in
$[-2,\beta-2]$, and
\[
 Z=\#\bigl([-2,\beta-2]\cap
                      [\delta,\delta+\epsilon]\bigr)
   =\epsilon+1.
\]
Combining the two sectors proves every coefficient in
\eqref{eq:coeff}; all other target shapes were excluded in
Step 2. Both possible nonzero coefficients are positive.
For $(\alpha,\beta,\gamma,\delta,\epsilon)=(k,0,0,0,2)$,
the first case gives the coefficient one of
$s_{(5^k,1,1)}$.
\end{proof}

\begin{corollary}[Reciprocity and positivity]
\label{cor:reciprocity}
On every positive $N$-letter alphabet, for $0\le k\le N-1$,
\begin{equation}\label{eq:reciprocity}
 \ResA_k(a)=\Pi^5\ResB_{N-k-1}(a^{-1}),\qquad
 \Pi=\prod_{i=1}^{N}a_i,
 \qquad a^{-1}=(a_1^{-1},\ldots,a_N^{-1}).
\end{equation}
Moreover $\ResB_k(a)\ge0$; and if
$k\ge1$ and $N\ge k+1$, then
\begin{equation}\label{eq:strict}
 \ResA_k(a)\ge s_{(5^{k-1},4,4)}(a)>0.
\end{equation}
\end{corollary}

\begin{proof}
Complementing subsets in the definition of the elementary
symmetric function gives, for $0\le h\le N$,
\[
 e_h(a)
  =\sum_{|S|=h}\prod_{i\in S}a_i
  =\Pi\sum_{|T|=N-h}\prod_{i\in T}a_i^{-1}
  =\Pi e_{N-h}(a^{-1}).
\]
For Schur functions, Lemma~\ref{lem:complement} states the
corresponding rectangular-complement identity: if
$\mu\subseteq(q^N)$, padded with zeros to $N$ parts, and
$\mu^c=(q-\mu_N,\ldots,q-\mu_1)$, then
\[
 \Pi^q s_\mu(a^{-1})=s_{\mu^c}(a).
\]
In particular the complement of $(q^h)$ is
$(q^{N-h})$, so
\begin{equation}\label{eq:rectangle-reciprocity}
 s_{(q^h)}(a)=\Pi^q s_{(q^{N-h})}(a^{-1})
 \qquad(0\le h\le N).
\end{equation}

Put $\ell=N-k-1$, so $0\le\ell\le N-1$.
Each deleted alphabet has $N-1$ letters and product
$\Pi/a_i$. Applying \eqref{eq:rectangle-reciprocity}
there with $q=4$ and $h=k$ gives
\[
 s_{(4^k)}(a_{-i})
 =\left(\frac{\Pi}{a_i}\right)^4
      s_{(4^{N-1-k})}((a^{-1})_{-i})
 =\Pi^4 a_i^{-4}s_{(4^\ell)}((a^{-1})_{-i}).
\]
Multiplication by $a_i^2$ and summation over $i$ yield
\begin{equation}\label{eq:F-reciprocity}
 F_k(a)
 =\Pi^4\sum_i a_i^{-2}
          s_{(4^\ell)}((a^{-1})_{-i})
 =\Pi^4F_\ell(a^{-1}).
\end{equation}
This includes $\ell=0$, when the Schur functions in the
sum are all equal to one.

The two factors in the first term defining $\ResA_k$
transform as
\[
 s_{(2^k)}(a)=\Pi^2s_{(2^{\ell+1})}(a^{-1}),\qquad
 s_{(3^{k+1})}(a)=\Pi^3s_{(3^\ell)}(a^{-1}).
\]
The factor in its second term satisfies
$e_{k+1}(a)=\Pi e_\ell(a^{-1})$.
Substitution of these three identities and
\eqref{eq:F-reciprocity} into \eqref{eq:scalars} gives
\[
\begin{split}
 \ResA_k(a)
 &=\Pi^5\left[
       s_{(3^\ell)}(a^{-1})s_{(2^{\ell+1})}(a^{-1})
            -e_\ell(a^{-1})F_\ell(a^{-1})\right]\\
 &=\Pi^5\ResB_\ell(a^{-1}),
\end{split}
\]
which proves \eqref{eq:reciprocity}, including the endpoints
$k=0$ and $k=N-1$.

Lemma~\ref{lem:scalar}, together with the tableau
nonnegativity in \eqref{eq:schur-tableaux}, proves
$\ResB_k(a)\ge0$ on every positive alphabet.
Now assume $k\ge1$. Then
$\ell+2=N-k+1\le N$, so the distinguished coefficient-one
partition $(5^\ell,1,1)$ fits into $(5^N)$.
Keeping this term in the nonnegative Schur expansion gives
\[
 \ResA_k(a)=\Pi^5\ResB_\ell(a^{-1})
       \ge\Pi^5s_{(5^\ell,1,1)}(a^{-1}).
\]
After padding $(5^\ell,1,1)$ with
$N-\ell-2=k-1$ zero parts, its complement in $(5^N)$
is
\[
 (\underbrace{5,\ldots,5}_{k-1},4,4,
                  \underbrace{0,\ldots,0}_{\ell}).
\]
Lemma~\ref{lem:complement} therefore identifies the last
quantity with $s_{(5^{k-1},4,4)}(a)$.
This partition has exactly $k+1\le N$ rows.
There is a semistandard tableau of this shape obtained by
filling every box in row $i$ with the entry $i$: entries
are weakly increasing along rows and strictly increasing down
columns. Its weight is positive because every $a_i>0$.
All other tableau weights are nonnegative, proving the strict
inequality in \eqref{eq:strict}.
\end{proof}

\section{Inertia of the deletion matrix}\label{sec:deletion-inertia}

\begin{lemma}[Comparison from a nonpositive hyperplane]\label{lem:hyperplane-comparison}
Let $C$ be a real symmetric matrix, and suppose that
$z^{\top}Cz\le0$ whenever $b^{\top}z=0$.
If $e^{\top}Ce=V>0$, then
\[
 C\preceq\frac{(Ce)(Ce)^{\top}}{V}.
\]
Moreover, $y^{\top}Cy\le0$ whenever $e^{\top}Cy=0$.
If the original hyperplane inequality is strict for nonzero $z$,
then this last inequality is strict for nonzero $y$.
\end{lemma}
\begin{proof}
We have $b^{\top}e\ne0$, since otherwise the hypothesis
applied to $e$ would contradict $e^{\top}Ce>0$.
For $e^{\top}Cy=0$, put
$t=(b^{\top}y)/(b^{\top}e)$ and $z=y-te$.
Then $b^{\top}z=0$, and symmetry gives
\[
 0\ge z^{\top}Cz=y^{\top}Cy+t^2V,
 \qquad y^{\top}Cy\le -t^2V\le0.
\]
If $y\ne0$, then $z\ne0$: otherwise $y=te$ and
$0=e^{\top}Cy=tV$ would force $t=0$ and $y=0$.
Thus strictness on the original hyperplane implies
$y^{\top}Cy<0$.
For an arbitrary $x$, now take
$y=x-(e^{\top}Cx/V)e$. Then $e^{\top}Cy=0$, so
\[
 0\ge y^{\top}Cy
   =x^{\top}Cx-\frac{(e^{\top}Cx)^2}{V}.
\]
Since $e^{\top}Cx=(Ce)^{\top}x$, this is the claimed
matrix inequality.
\end{proof}

\begin{lemma}\label{lem:inertia}
Let $a=(a_1,\ldots,a_n)$ be a positive alphabet, where $n\ge2$,
and let $0\le j\le n-2$. The real symmetric matrix
\[
 H_{i\ell}=
 \begin{cases}
 K_j(a_{-i\ell}),&i\ne\ell,\\
 0,&i=\ell
 \end{cases}
\]
admits a vector $b$ such that $z^{\top}Hz<0$ for every
nonzero $z\in b^\perp$. It has exactly one positive eigenvalue;
all its other eigenvalues are negative, with multiplicities counted.
\end{lemma}

\begin{proof}
For $j=0$, every off-diagonal entry is one. Thus
$H=\1_n\1_n^{\top}-I$, whose eigenvalue on
$\operatorname{span}\{\1_n\}$ is $n-1$ and whose restriction
to $\1_n^\perp$ is $-I$.

For the rest of the proof, let $j\ge1$ and put $r=j+1$.
Consequently $2\le r\le n-1$. Define
\[
 b_i=K_{r-1}(a_{-i}),\qquad
 t_i=a_i\Srect_r(a_{-i}),\qquad
 R=bb^{\top}-K_{r-1}(a)H
       -\frac{e_r(a)}{\Srect_{r+1}(a)}tt^{\top}.
\]
Here $K_{r-1}(a)=s_{(2^{r-1})}(a)>0$,
$e_r(a)>0$, and
$\Srect_{r+1}(a)=s_{(3^r)}(a)>0$: in each case the partition
has at most $n$ rows. We shall prove that $R$ is positive definite.

\paragraph{The rectangular condensation identity.}
Let $A$ be a finite alphabet and let $x,y$ be two further letters.
Put $d=r-1\ge1$ and consider the four-by-four dual
Jacobi-Trudi matrix
\[
 T_{i\ell}=e_{d-i+\ell}(A),\qquad 1\le i,\ell\le4.
\]
By \eqref{eq:dual-jt}, $\det T=s_{(4^d)}(A)$.
For a letter $z$, set
\[
 U(z)=\begin{pmatrix}
 1&z&0&0\\0&1&z&0\\0&0&1&z\\0&0&0&1
 \end{pmatrix},\qquad
 \widetilde T=U(x)TU(y).
\]
Both triangular factors have determinant one, so
$\det\widetilde T=\det T$.
Left multiplication adds $x$ times the next original row;
right multiplication adds $y$ times the preceding original
column. The elementary identity
$e_q(z,A)=e_q(A)+ze_{q-1}(A)$, valid also for negative
indices under our zero convention, therefore gives
\begin{equation}\label{eq:transformed-jt}
 \widetilde T_{i\ell}=
 \begin{cases}
 e_{d-i+1}(x,A),&i<4,\ \ell=1,\\
 e_{d-i+\ell}(x,y,A),&i<4,\ \ell\ge2,\\
 e_{d-3}(A),&i=4,\ \ell=1,\\
 e_{d-4+\ell}(y,A),&i=4,\ \ell\ge2.
 \end{cases}
\end{equation}
Write $\widetilde T_{\widehat I,\widehat J}$ for the matrix
obtained by deleting the rows in $I$ and the columns in $J$.
The five minors needed for \eqref{eq:desnanot-jacobi} are
\begin{align*}
 \det\widetilde T_{\widehat4,\widehat4}
   &=s_{(3^d)}(x,A),&
 \det\widetilde T_{\widehat1,\widehat1}
   &=s_{(3^d)}(y,A),\\
 \det\widetilde T_{\widehat1,\widehat4}
   &=s_{(3^{d-1})}(A),&
 \det\widetilde T_{\widehat4,\widehat1}
   &=s_{(3^{d+1})}(x,y,A),\\
 \det\widetilde T_{\widehat{\{1,4\}},\widehat{\{1,4\}}}
   &=s_{(2^d)}(x,y,A).
\end{align*}
In the upper-left three-by-three minor, perform the row additions
first. All its rows then have entries $e_{d-i+\ell}(x,A)$;
the subsequent column additions are internal to this minor and
preserve its determinant. This gives its stated value by
\eqref{eq:dual-jt}. For the lower-right minor, perform the column
additions first: its columns have entries
$e_{d-i+\ell}(y,A)$, and the subsequent row additions are
internal. Subtracting one from both indices gives the second value.

In the lower-left minor, both sets of additions are internal.
Its determinant thus agrees with that of the corresponding minor
of $T$, whose entries after reindexing are
$e_{d-1-i+\ell}(A)$. The upper-right minor and the center
minor consist entirely of entries evaluated on $(x,y,A)$;
reindexing gives respectively
$e_{d+1-i+\ell}(x,y,A)$ and
$e_{d-i+\ell}(x,y,A)$. These are precisely the remaining
three dual Jacobi-Trudi determinants above. When $d=1$,
the lower-left determinant is triangular with diagonal entries
$e_0=1$, and equals $s_{(3^0)}=1$.

Applying \eqref{eq:desnanot-jacobi} to $\widetilde T$, then
using $d=r-1$ and \eqref{eq:K-S-definitions}, gives
\begin{equation}\label{eq:condensation}
 \Srect_r(x,A)\Srect_r(y,A)
 -\Srect_{r+1}(x,y,A)\Srect_{r-1}(A)
 =K_{r-1}(x,y,A)s_{(4^{r-1})}(A).
\end{equation}
\paragraph{The elementary two-deletion identity.}
We next prove, for $j\ge1$,
\begin{equation}\label{eq:two-deletion}
 K_j(x,A)K_j(y,A)-K_j(x,y,A)K_j(A)
 =xy\,e_{j+1}(x,y,A)\Srect_j(A).
\end{equation}
All elementary symmetric functions in the following calculation
are evaluated on $A$. Set
\[
 u=e_{j+1},\quad v=e_j,\quad z=e_{j-1},\quad
 d=e_{j-2},\quad f=e_{j-3},
\]
and define
\[
 c_0=v^2-uz,\qquad
 c_1=vz-ud,\qquad
 c_2=z^2-vd,
\]
\[
 q_0=2vd-z^2-uf,\qquad
 q_1=zd-vf,\qquad
 q_2=d^2-zf.
\]
In this calculation, $f$ and $c_i$ denote the quantities above. By \eqref{eq:dual-jt},
\begin{equation}\label{eq:three-square-expansion}
 T\coloneqq \Srect_j(A)
  =\det\begin{pmatrix}
  z&v&u\\ d&z&v\\ f&d&z
  \end{pmatrix}
  =z^3-2vzd+v^2f+ud^2-uzf.
\end{equation}
This includes $j=1$: the zero-index and negative-index
conventions make the displayed determinant equal to one.

Since $e_q(x,A)=e_q+xe_{q-1}$, direct substitution gives
\[
 K_j(x,A)=c_0+xc_1+x^2c_2.
\]
Writing $s=x+y$ and $p=xy$, the relation
$e_q(x,y,A)=e_q+se_{q-1}+pe_{q-2}$ similarly gives
\[
 K_j(x,y,A)=c_0+sc_1+s^2c_2+pq_0+spq_1+p^2q_2.
\]
On multiplying the two one-letter expressions, one obtains
\[
 K_j(x,A)K_j(y,A)
 =c_0^2+c_0c_1s+c_0c_2(s^2-2p)
   +c_1^2p+c_1c_2ps+c_2^2p^2.
\]
Subtracting $c_0K_j(x,y,A)$, with $K_j(A)=c_0$, leaves
\begin{equation}\label{eq:two-deletion-coefficients}
 p\left[
 c_1^2-c_0(2c_2+q_0)
 +s(c_1c_2-c_0q_1)+p(c_2^2-c_0q_2)
 \right].
\end{equation}
The three coefficients in brackets are, respectively,
$uT,vT,zT$. Indeed, $2c_2+q_0=z^2-uf$, and
\begin{align}
 (vz-ud)^2-(v^2-uz)(z^2-uf)
   &=u(z^3-2vzd+v^2f+ud^2-uzf)=uT,\nonumber\\
 (vz-ud)(z^2-vd)-(v^2-uz)(zd-vf)
   &=v(z^3-2vzd+v^2f+ud^2-uzf)=vT,\label{eq:minor-coefficients}\\
 (z^2-vd)^2-(v^2-uz)(d^2-zf)
   &=z(z^3-2vzd+v^2f+ud^2-uzf)=zT.\nonumber
\end{align}
Inserting these into \eqref{eq:two-deletion-coefficients} gives
$p(u+sv+pz)T$.
Since $u+sv+pz=e_{j+1}(x,y,A)$, this proves
\eqref{eq:two-deletion}.

\paragraph{The off-diagonal entries of $R$.}
Fix $i\ne\ell$, and put
$x=a_i$, $y=a_\ell$, and $A=a_{-i\ell}$.
The definition of $R$ and \eqref{eq:two-deletion},
with $j=r-1$, give
\begin{align*}
 R_{i\ell}
 &=K_{r-1}(y,A)K_{r-1}(x,A)
   -K_{r-1}(x,y,A)K_{r-1}(A)\\
 &\hspace{12mm}
   -\frac{e_r(x,y,A)}{\Srect_{r+1}(x,y,A)}
       xy\Srect_r(y,A)\Srect_r(x,A)\\
 &=xy e_r(x,y,A)\left[
   \Srect_{r-1}(A)
   -\frac{\Srect_r(x,A)\Srect_r(y,A)}
          {\Srect_{r+1}(x,y,A)}
   \right].
\end{align*}
Equation \eqref{eq:condensation} therefore yields
\begin{equation}\label{eq:off}
 R_{i\ell}
 =-\frac{K_{r-1}(a)e_r(a)}{\Srect_{r+1}(a)}
       a_i a_\ell s_{(4^{r-1})}(a_{-i\ell})\le0
       \qquad(i\ne\ell).
\end{equation}
The prefactors are positive, and the tableau expansion makes the
Schur value nonnegative.

\paragraph{The adjacent-minor identities.}
To compute the diagonal, write $a=(x,A)$ and abbreviate
\[
 c=K_{r-1}(A),\quad C=K_r(A),\quad
 T=\Srect_r(A),\quad U=\Srect_{r+1}(A),\quad
 f=K_{r-1}(x,A),\quad g=e_r(x,A).
\]
We shall establish
\begin{align}
 fU&=c\Srect_{r+1}(x,A)-xK_r(x,A)T,
       \label{eq:adjacent-first}\\
 cK_r(x,A)-fC&=xgT.
       \label{eq:adjacent-second}
\end{align}
Every $e_q$ in the following minors is evaluated on $A$.
Form the four row vectors
\[
 v_i=(e_{r-i},e_{r+1-i},e_{r+2-i}),
 \qquad i=0,1,2,3.
\]
For increasing indices, set
\[
 \Delta_{ijk}=\det\begin{pmatrix}v_i\\v_j\\v_k\end{pmatrix},
 \qquad
 L_{ij}=\det\begin{pmatrix}
 e_{r-i}&e_{r+1-i}\\
 e_{r-j}&e_{r+1-j}
 \end{pmatrix}.
\]
Dual Jacobi-Trudi identifies
\[
 U=\Delta_{012},\qquad T=\Delta_{123},\qquad
 c=L_{12},\qquad C=L_{01}.
\]

The alternating-minor identity
\begin{equation}\label{eq:four-row-relation}
 v_0\Delta_{123}-v_1\Delta_{023}
 +v_2\Delta_{013}-v_3\Delta_{012}=0
\end{equation}
follows by expanding a determinant with a repeated column.
For each $q\in\{1,2,3\}$, form the four-by-four matrix whose
row indexed by $i$ is
\[
 \bigl((v_i)_q,(v_i)_1,(v_i)_2,(v_i)_3\bigr),
 \qquad i=0,1,2,3.
\]
Its first column equals column $q+1$, so its determinant is zero.
Expanding along the first column gives the $q$-th coordinate of
\eqref{eq:four-row-relation}.

Let $w_i$ be the first two coordinates of $v_i$.
Apply the linear functional $w\mapsto\det(w,w_1)$ to the first
two coordinates of \eqref{eq:four-row-relation}. This gives
\[
 L_{01}\Delta_{123}-L_{12}\Delta_{013}
   +L_{13}\Delta_{012}=0.
\]
Using instead $w\mapsto\det(w,w_2)$ gives
\[
 L_{02}\Delta_{123}-L_{12}\Delta_{023}
   +L_{23}\Delta_{012}=0.
\]
Thus
\begin{equation}\label{eq:adjacent-minor-relations}
 \begin{split}
 L_{12}\Delta_{013}-L_{01}T&=L_{13}U,\\
 L_{12}\Delta_{023}-L_{02}T&=L_{23}U.
 \end{split}
\end{equation}

The one-letter relation for $e_q$ makes the three rows of
the determinant for $\Srect_{r+1}(x,A)$ equal to
$v_0+xv_1,v_1+xv_2,v_2+xv_3$.
Multilinearity of the determinant, and the vanishing of every
term with repeated rows, give
\begin{equation}\label{eq:one-letter-minors}
 \begin{split}
 \Srect_{r+1}(x,A)
   &=U+x\Delta_{013}+x^2\Delta_{023}+x^3T,\\
 K_{r-1}(x,A)&=L_{12}+xL_{13}+x^2L_{23},\\
 K_r(x,A)&=L_{01}+xL_{02}+x^2L_{12}.
 \end{split}
\end{equation}
For the two-by-two formulas the same multilinear expansion uses
the first two coordinates of, respectively,
$(v_1+xv_2,v_2+xv_3)$ and
$(v_0+xv_1,v_1+xv_2)$.
Consequently
\begin{align*}
 c\Srect_{r+1}(x,A)-xK_r(x,A)T
 &=cU+x(c\Delta_{013}-L_{01}T)
       +x^2(c\Delta_{023}-L_{02}T)\\
 &=(L_{12}+xL_{13}+x^2L_{23})U=fU,
\end{align*}
where the cubic terms have canceled and
\eqref{eq:adjacent-minor-relations} supplies the second line.
This proves \eqref{eq:adjacent-first}.

To prove \eqref{eq:adjacent-second}, the last two coefficient
identities in \eqref{eq:minor-coefficients}, applied with index
$j=r$ on $A$, give
\[
 L_{02}L_{12}-L_{13}L_{01}=e_r(A)T,\qquad
 L_{12}^2-L_{23}L_{01}=e_{r-1}(A)T.
\]
Indeed the identifications in that calculation are
\[
 L_{01}=v^2-uz,\quad L_{02}=vz-ud,\quad
 L_{12}=z^2-vd,\quad L_{13}=zd-vf,\quad L_{23}=d^2-zf,
\]
with $(u,v,z,d,f)=(e_{r+1},e_r,e_{r-1},e_{r-2},e_{r-3})$.
Expanding with \eqref{eq:one-letter-minors} now gives
\begin{align*}
 cK_r(x,A)-fC
 &=x(L_{12}L_{02}-L_{13}L_{01})
     +x^2(L_{12}^2-L_{23}L_{01})\\
 &=x\bigl(e_r(A)+xe_{r-1}(A)\bigr)T
  =xgT.
\end{align*}
This proves the second adjacent identity, including the cases
where some of the displayed elementary indices are negative.

Multiplying \eqref{eq:adjacent-first} by $c$ and substituting
$cK_r(x,A)=fC+xgT$ from \eqref{eq:adjacent-second}, we find
\[
 cfU=c^2\Srect_{r+1}(x,A)-x(fC+xgT)T.
\]
Equivalently,
\[
 c^2\Srect_{r+1}(x,A)-x^2gT^2=f(cU+xCT).
\]
Since $H_{ii}=0$, the defining diagonal entry of $R$ is
$c^2-x^2gT^2/\Srect_{r+1}(x,A)$. Dividing the last identity
by the positive denominator therefore proves
\begin{equation}\label{eq:diag}
 R_{ii}
 =\frac{K_{r-1}(a)}{\Srect_{r+1}(a)}
 \left[
 K_{r-1}(a_{-i})\Srect_{r+1}(a_{-i})
 +a_i\Srect_r(a_{-i})K_r(a_{-i})
 \right].
\end{equation}

\paragraph{A positive vector and a sum of squares.}
Fix $i$, write $A=a_{-i}$, and put $k=r-1$.
By the definition of $F_k$,
\[
 F_k(A)=\sum_{\ell\ne i}
       a_\ell^2s_{(4^k)}(a_{-i\ell}).
\]
Multiplying \eqref{eq:diag} by $a_i$, and multiplying each
off-diagonal entry in row $i$ of \eqref{eq:off} by its
corresponding $a_\ell$, yields
\begin{align*}
 (Ra)_i
 &=\frac{K_k(a)a_i}{\Srect_{r+1}(a)}
 \left[
 K_k(A)\Srect_{r+1}(A)
 +a_i\Srect_r(A)K_r(A)-e_r(a)F_k(A)
 \right]\\
 &=\frac{K_k(a)a_i}{\Srect_{r+1}(a)}
 \left[
 K_k(A)\Srect_{r+1}(A)-e_r(A)F_k(A)
 +a_i\bigl(\Srect_r(A)K_r(A)-e_k(A)F_k(A)\bigr)
 \right]\\
 &=\frac{K_{r-1}(a)a_i}{\Srect_{r+1}(a)}
       \left[\ResA_{r-1}(a_{-i})
                  +a_i\ResB_{r-1}(a_{-i})\right].
\end{align*}
The second line uses $e_r(a)=e_r(A)+a_i e_{r-1}(A)$;
the third uses the scalar definitions \eqref{eq:scalars}.
Here $k=r-1\ge1$, and $A$ has $N=n-1\ge r=k+1$
positive letters. Corollary~\ref{cor:reciprocity} therefore gives
$\ResA_k(A)>0$ and $\ResB_k(A)\ge0$. Every prefactor is
positive, so
\begin{equation}\label{eq:R-positive-vector}
 (Ra)_i>0\qquad\text{for every }i.
\end{equation}

For any real vector $z$, consider
\[
 \sum_{i<\ell}(-R_{i\ell})a_i a_\ell
 \left(\frac{z_i}{a_i}-\frac{z_\ell}{a_\ell}\right)^2
 +\sum_i\frac{(Ra)_i}{a_i}z_i^2.
\]
The coefficient of $z_i z_\ell$, for $i<\ell$, in this
expression is $2R_{i\ell}$. Its coefficient of $z_i^2$ is
\[
 \sum_{\ell\ne i}(-R_{i\ell})\frac{a_\ell}{a_i}
       +\frac{(Ra)_i}{a_i}
 =-\sum_{\ell\ne i}R_{i\ell}\frac{a_\ell}{a_i}
       +R_{ii}+\sum_{\ell\ne i}R_{i\ell}\frac{a_\ell}{a_i}
 =R_{ii}.
\]
Thus the expression equals $z^{\top}Rz$, proving
\begin{equation}\label{eq:R-squares}
 z^{\top}Rz
 =\sum_{i<\ell}(-R_{i\ell})a_i a_\ell
 \left(\frac{z_i}{a_i}-\frac{z_\ell}{a_\ell}\right)^2
 +\sum_i\frac{(Ra)_i}{a_i}z_i^2.
\end{equation}
Each term is nonnegative by \eqref{eq:off} and
\eqref{eq:R-positive-vector}. If $z\ne0$, some $z_i^2>0$,
and its term in the second sum is strictly positive.
Therefore $R$ is positive definite.

Finally, for a nonzero $z\in b^\perp$, the definition of $R$
gives
\[
 K_{r-1}(a)z^{\top}Hz
 =-\frac{e_r(a)}{\Srect_{r+1}(a)}(t^{\top}z)^2
       -z^{\top}Rz<0.
\]
Thus $H$ is negative definite on $b^\perp$.
Also,
\[
 \1_n^{\top}H\1_n
  =2\sum_{i<\ell}K_j(a_{-i\ell})>0,
\]
because $K_j=s_{(2^j)}$, $j\le n-2$, and every deleted
alphabet has $n-2$ positive letters. By the spectral theorem,
$H$ therefore has a unit eigenvector $p$ with positive
eigenvalue. In particular $p^{\top}Hp>0$.
Apply the strict conclusion of Lemma~\ref{lem:hyperplane-comparison}
with $C=H$ and $e=p$. Every other vector $q$ in an
orthonormal eigenbasis satisfies
$p^{\top}Hq=\lambda_q p^{\top}q=0$, hence
$\lambda_q=q^{\top}Hq<0$. This proves the assertion,
including multiplicities.
\end{proof}

\section{Covariance bound and sharpness}\label{sec:covariance-proof}

\begin{proof}[Proof of Theorem~\ref{thm:main}]
Write $Z=e_m(w)>0$. Summing the probabilities of all
$m$-subsets containing one or two specified indices gives
\[
 \mathbb E X_i=\frac{w_i e_{m-1}(w_{-i})}{Z},\qquad
 \mathbb E(X_iX_\ell)=
       \frac{w_iw_\ell e_{m-2}(w_{-i\ell})}{Z}
       \quad(i\ne\ell).
\]
For $m=1$, the second formula is zero by the convention
$e_{-1}=0$, as required.

To simplify the covariance numerator, fix $i\ne\ell$ and set
\[
 x=w_i,\quad y=w_\ell,\quad A=w_{-i\ell},\qquad
 E=e_{m-1}(A),\quad F=e_{m-2}(A),\quad G=e_m(A).
\]
Then
\[
 e_{m-1}(w_{-i})=E+yF,\qquad
 e_{m-1}(w_{-\ell})=E+xF,\qquad
 Z=G+(x+y)E+xyF.
\]
Consequently
\begin{align*}
 -\Sigma_{i\ell}
 &=\frac{xy}{Z^2}
   \left[(E+yF)(E+xF)-F\bigl(G+(x+y)E+xyF\bigr)\right]\\
 &=\frac{xy}{Z^2}(E^2-FG)
  =\frac{w_iw_\ell}{e_m(w)^2}K_{m-1}(w_{-i\ell})>0.
\end{align*}
The last inequality follows from $m-1\le d-2$ and
the strict positivity of $s_{(2^{m-1})}$ on the
$d-2$ positive letters of $w_{-i\ell}$.

Since $\sum_iX_i=m$ with probability one, for each $i$ we have
\[
 (\Sigma\1_d)_i
 =\operatorname{Cov}\left(X_i,\sum_\ell X_\ell\right)
 =\operatorname{Cov}(X_i,m)=0.
\]
Put $c_{i\ell}=-\Sigma_{i\ell}>0$ for $i\ne\ell$.
The row-sum identity yields
\[
 v_i=\Sigma_{ii}=\sum_{\ell\ne i}c_{i\ell}>0.
\]
In particular $D^{-1/2}$ is defined and $V=\sum_i v_i>0$.
For later use, expansion of the quadratic form gives
\begin{equation}\label{eq:covariance-laplacian}
 z^{\top}\Sigma z
 =\sum_{i<\ell}c_{i\ell}(z_i-z_\ell)^2.
\end{equation}
All $c_{i\ell}$ are strictly positive, so this vanishes exactly
when all coordinates of $z$ are equal. Hence
$\ker\Sigma=\operatorname{span}\{\1_d\}$.

Set $C=D-\Sigma$, and let $H$ be the matrix from
Lemma~\ref{lem:inertia} evaluated at $a=w$ and $j=m-1$.
Its permitted range holds because $1\le m\le d-1$.
The diagonal of $C$ is zero and its off-diagonal entries are
$c_{i\ell}$, so the covariance formula above gives
\begin{equation}\label{eq:covariance-congruence}
 C=Z^{-2}\diag(w)H\diag(w),\qquad C\1_d=v.
\end{equation}
Let $b$ be the hyperplane normal supplied by
Lemma~\ref{lem:inertia}, and set $\widetilde b=\diag(w)b$.
If $\widetilde b^{\top}z=0$, then
$b^{\top}\diag(w)z=0$, and \eqref{eq:covariance-congruence}
gives
\[
 z^{\top}Cz
 =Z^{-2}(\diag(w)z)^{\top}H(\diag(w)z)\le0.
\]
Since $C\1_d=v$ and $\1_d^{\top}C\1_d=V>0$,
Lemma~\ref{lem:hyperplane-comparison}, with $e=\1_d$, yields
$C\preceq vv^{\top}/V$. Substituting $C=D-\Sigma$ gives
\[
 \Sigma\succeq D-\frac{vv^{\top}}{V}.
\]
For the spectral formulation, put
$N=D^{-1/2}\Sigma D^{-1/2}$ and
$u=(\sqrt{v_1},\ldots,\sqrt{v_d})^{\top}$, so
$\|u\|^2=V$.
The kernel calculation following \eqref{eq:covariance-laplacian}
shows that $\ker N=\operatorname{span}\{u\}$.
The displayed matrix inequality is equivalent by congruence to
\[
 N\succeq I-\frac{uu^{\top}}{V}.
\]
The right-hand side is zero on $\operatorname{span}\{u\}$
and is the identity on $u^\perp$.
Both subspaces are invariant under the symmetric matrix $N$.
Thus this inequality is equivalent to every eigenvalue of
$N$ on $u^\perp$, namely every nonzero eigenvalue, being at
least one. The cases $m=1$ and $d=2$ correspond to
$j=m-1=0$ in Lemma~\ref{lem:inertia}.

It remains to show optimality. Take $d=3,m=1$ and choose weights
equal to the probabilities
\[
 p_1=\varepsilon,\qquad
 p_2=p_3=q=\frac{1-\varepsilon}{2},\qquad
 0<\varepsilon<1.
\]
These positive weights sum to one, so they realize the stated
probabilities in the product-weight model.
Here $\Sigma=\diag(p)-pp^{\top}$, with
$v_i=p_i(1-p_i)$.
The normalized covariance $N$ therefore has diagonal entries
one and off-diagonal entries
$-p_i p_\ell/\sqrt{v_i v_\ell}$.
Because $p_2=p_3$ and $v_2=v_3$, the vector
$(0,1,-1)^{\top}$ is an eigenvector:
its first coordinate after multiplication is zero, and its last
two coordinates are multiplied by
\[
 1+\frac{q^2}{q(1-q)}
 =\frac1{1-q}=\frac{2}{1+\varepsilon}.
\]
There is also the zero eigenvalue, with eigenvector $\sqrt v$.
The trace is three, since every diagonal entry of $N$ is one,
so its remaining eigenvalue is
\[
 3-\frac{2}{1+\varepsilon}
 =\frac{1+3\varepsilon}{1+\varepsilon}.
\]
As $\varepsilon\downarrow0$, this last nonzero eigenvalue tends
to one. For every proposed universal coefficient strictly larger
than one, sufficiently small positive $\varepsilon$ therefore
violates that coefficient. The coefficient one is optimal.
\end{proof}

\paragraph{Strictness for positive weights.}
Put
\[
 N=D^{-1/2}\Sigma D^{-1/2},\qquad
 P=D-\frac{vv^{\top}}V,\qquad
 u=(\sqrt{v_1},\ldots,\sqrt{v_d})^{\top},\qquad M=I-N.
\]
By \eqref{eq:covariance-congruence}, $M$ is congruent to the
matrix $H$ of Lemma~\ref{lem:inertia} with $j=m-1$.
It therefore has one positive eigenvalue and $d-1$ negative
eigenvalues. Since $\Sigma\1_d=0$, we have $Mu=u$, so its
positive eigenvalue is one. Consequently every nonzero eigenvalue
of $N$ is strictly greater than one.
Also $\operatorname{tr}N=d$, because all its diagonal entries
are one. For $d\ge3$, the other $d-2$ nonzero eigenvalues
have sum greater than $d-2$, so each nonzero eigenvalue is
strictly less than two. Thus
\[
 \operatorname{spec}(N)\subseteq\{0\}\cup(1,2)\quad(d\ge3),
 \qquad \operatorname{spec}(N)=\{0,2\}\quad(d=2).
\]
The sharpness family above approaches the lower endpoint one.

Moreover,
\[
 D^{-1/2}(\Sigma-P)D^{-1/2}
 =N-I+\frac{uu^{\top}}V
\]
vanishes on $\operatorname{span}\{u\}$ and is positive definite
on $u^\perp$. Since $D^{-1/2}u=\1_d$, this gives
\[
 \ker(\Sigma-P)=\operatorname{span}\{\1_d\}.
\]
\section{Corollaries}\label{sec:corollaries}

\subsection{Pseudoinverse bound}

For an eigenvector $u$ of $\Sigma$ with eigenvalue $\lambda\ge0$,
\[
    \Sigma^\dagger u
=
    \begin{cases}
        \lambda^{-1}u,  & \lambda>0,\\
        0,              & \lambda=0.
    \end{cases}
\]

The pseudoinverse bound of \citet[Lemma~2]{bacchiocchicesaricolomboni2026} applies to the same
positive weighted $m$-set laws and has coefficient $2$.
Corollary~\ref{cor:inverse} improves that coefficient to $1$.

\begin{proof}[Proof of Corollary~\ref{cor:inverse}]
Put $P=D-vv^{\top}/V$.
Weighted Cauchy-Schwarz gives $P\succeq0$, with
$\ker P=\operatorname{span}\{\1_d\}$, and
Theorem~\ref{thm:main} gives $\Sigma-P\succeq0$.
The kernel computation following \eqref{eq:covariance-laplacian}
and symmetry give
\[
 \ker\Sigma=\operatorname{span}\{\1_d\},\qquad
 \operatorname{im}\Sigma=\1_d^\perp.
\]
For $y\in\1_d^\perp$, set $h=\Sigma^\dagger y$ and
$z=D^{-1}y$. Then $\Sigma h=y$, and
$\1_d^{\top}y=0$ gives $Pz=y$.
Expanding the right-hand side yields
\[
 y^{\top}D^{-1}y-y^{\top}\Sigma^\dagger y
 =(z-h)^{\top}P(z-h)
   +h^{\top}(\Sigma-P)h\ge0,
\]
which proves \eqref{eq:inverse-corollary}.
If $y\ne0$, then $h\in\1_d^\perp\setminus\{0\}$.
The preceding strictness argument gives
$\ker(\Sigma-P)=\operatorname{span}\{\1_d\}$, so the last
term is positive and \eqref{eq:inverse-corollary} is strict.

For optimality, put $E=\1_d^\perp$. The restrictions
$\Sigma_E$ and $P_E$ to $E$ are positive definite.
For $y\in E$, the identities $PD^{-1}y=y=PP^\dagger y$
show that $D^{-1}y-P^\dagger y\in\ker P$. Hence
\[
 y^{\top}P^\dagger y=y^{\top}D^{-1}y.
\]
A coefficient $c\le0$ is impossible because
$y^{\top}\Sigma^\dagger y>0$ for nonzero $y\in E$.
If some $0<c<1$ satisfied the bound for every admissible law,
the preceding identity would give
$\Sigma_E^{-1}\preceq cP_E^{-1}$.
Inversion reverses the order of positive definite matrices, so
$\Sigma_E\succeq c^{-1}P_E$.
Since $\Sigma\1_d=P\1_d=0$, we obtain
$\Sigma\succeq c^{-1}P$ on the whole space for every such law.
Since $c^{-1}>1$, this contradicts the sharpness of
Theorem~\ref{thm:main}.
\end{proof}

\subsection{Effective resistance}\label{sec:resistance-details}

Taking $y=\mathbf e_i-\mathbf e_j$ in Corollary~\ref{cor:inverse} recovers the
effective-resistance theorem proved in the first version of this paper
\citep[Theorem~1.1]{cesaricolomboni2026}, with the same right-hand side.

By \eqref{eq:covariance-laplacian}, $\Sigma$ is the Laplacian
of the complete graph on $[d]$ with conductances
$c_{ij}\coloneqq -\Sigma_{ij}>0$ for $i\ne j$.
The left-hand side of \eqref{eq:resistance-corollary} is therefore
the effective resistance between vertices $i$ and $j$.

\begin{proof}[Proof of Corollary~\ref{cor:resistance}]
Apply Corollary~\ref{cor:inverse} with $y=\mathbf e_i-\mathbf e_j$.
\end{proof}

\paragraph{Electrical interpretation.}
For $i\ne j$, the vector $\mathbf e_i-\mathbf e_j$ belongs to $\1_d^\perp$.
Since $\Sigma$ is symmetric and $\ker\Sigma=\Span\{\1_d\}$, we have $\operatorname{im}\Sigma=\1_d^\perp$.
Hence the equation
\[
    \Sigma h
=
    \mathbf e_i-\mathbf e_j
\]
has a unique solution $h\in\1_d^\perp$, namely $h=\Sigma^\dagger(\mathbf e_i-\mathbf e_j)$.
By the definition of $R_{ij}$,
\[
 R_{ij}=(\mathbf e_i-\mathbf e_j)^\top h=h_i-h_j.
\]
Interpret $c_{ab}$ as the conductance of the edge $\{a,b\}$, so that the resistance of this edge is $1/c_{ab}$.
For a potential vector $h$, the $k$th coordinate of $\Sigma h$ is
$
    (\Sigma h)_k
=
    \sum_{\ell\ne k} c_{k\ell}(h_k-h_\ell),
$
which is the net current injected at vertex $k$.
Therefore, the equation
$
  \Sigma h=\mathbf e_i-\mathbf e_j
$
means that one unit of current is injected at vertex $i$, one unit is extracted at vertex $j$, and the net current at every other vertex is zero.
Since adding the same constant to all coordinates of $h$ does not change any voltage difference, this equation determines $h$ only up to an additive constant.
Requiring $h\in\1_d^\perp$, or equivalently $\1_d^\top h=0$, selects the unique solution whose coordinates have mean zero.
The voltage difference between vertices $i$ and $j$ is
$
    h_i-h_j.
$
By definition, the effective resistance between two terminals is the voltage difference required to send one unit of current from one terminal to the other.
Since the equation
$
    \Sigma h=\mathbf e_i-\mathbf e_j
$
describes one unit of current injected at $i$ and extracted at $j$, we have
$
    R_{ij}=h_i-h_j.
$

\paragraph{Graph-theoretic interpretation and scope of the bound.}
The scale $1/v_i$ already appears in the simplest graph reductions.
If all vertices other than $i$ are identified, the resulting two-vertex graph has a single edge of total weight $\sum_{j\ne i}c_{ij}=v_i$.
Evaluating the definition of effective resistance on this graph gives
\[
    \left(\sum_{j\ne i}c_{ij}\right)^{-1}=\frac1{v_i}.
\]
For the two-edge path $i-k-j$ with edge weights $a$ and $b$, a direct calculation from the Laplacian gives
\[
    R_{ij}
=
    \frac1a+\frac1b
=
    \frac1{v_i}+\frac1{v_j}.
\]
Thus, \eqref{eq:resistance-corollary} is a one-sided extension of this exact identity.
The resulting upper bound is nevertheless false for general weighted graphs.
For example, fix $d\ge4$, put edge weight $1$ on the edges of the path $1-2-\cdots-d$, and put edge weight $\varepsilon$ on every remaining edge.
As $\varepsilon\downarrow0$, the effective resistance between vertices $1$ and $d$ converges to $d-1$, whereas
\[
    v_1
=
    v_d
=
    1+(d-2)\varepsilon,
\qquad
    v_1^{-1}+v_d^{-1} \to 2.
\]
Therefore, the bound depends on special characteristics of the fixed-rank covariance Laplacians, and not just the completeness of the given weighted graph.
This extra structure appears for example looking at the relative probability of exchanging $j$ for $i$, which turns out to be independent of the other selected elements: for every $T\subseteq[d]\setminus\{i,j\}$ with $|T|=m-1$,
\[
    \frac{\mathbb P(\mathsf S=T\cup\{i\})}{\mathbb P(\mathsf S=T\cup\{j\})}
=
    \frac{w_i}{w_j}.
\]
This identity shows that, after fixing a background set $T$, the relative likelihood of selecting $i$ rather than $j$ is always $w_i/w_j$.
Thus the choice between $i$ and $j$ can be analyzed separately from the randomness of the considered background set.
For comparison, \citet[Theorems~3-4 and Proposition~5]{vonLuxburgRadlHein2014} prove that, for several classes of large random neighborhood graphs, effective resistance is asymptotic to the sum of the inverse endpoint degrees.
We remark that their results require specific quantitative connectivity conditions.
On the other hand, \eqref{eq:resistance-corollary} gives an exact bound for every dimension, rank, and positive external field.

\paragraph{Asymptotic sharpness.}

In this paragraph, write $p=m/d$ and $v=p(1-p)$ for the common coordinate variance.
Then, for the uniform rank-$m$ law on $[d]$, we have that
\[
    \Sigma
=
    \frac{d}{d-1}v\left(I-\frac1d\1_d\1_d^\top\right),
\qquad
    R_{ij}
=
    \frac{2(d-1)}{dv}.
\]
The right-hand side of \eqref{eq:resistance-corollary} is $2/v$, so
\[
    \frac{R_{ij}}{v_i^{-1}+v_j^{-1}}
=
    \frac{d-1}{d} \to 1.
\]
Thus, the coefficient $1$ in \eqref{eq:resistance-corollary} is, at least asymptotically, the best possible.

\subsection{Correlated sampling}

\begin{proof}[Proof of Corollaries~\ref{cor:rounding} and~\ref{cor:rounding-atmost}]
We verify Conjecture~3 of \citet{anarihaqima2026}, version~2.
Let $C$ be the covariance of a positive external-field
$r$-subset law on $R$, indexed in increasing order, with $0<r<|R|$, and put
$D_C=\diag(C_{ii}:i\in R)$.
The calculation leading to \eqref{eq:covariance-laplacian} gives
$C\1_{|R|}=0$, $C_{ij}<0$ for $i\ne j$,
$C_{ii}>0$, and $\ker C=\operatorname{span}\{\1_{|R|}\}$.
Moreover, the same Laplacian identity gives
\[
 z^{\top}Cz
 =\sum_{i<j}(-C_{ij})(z_i-z_j)^2
 \le2\sum_i C_{ii}z_i^2,
\]
so $C\preceq2D_C$.
Theorem~\ref{thm:main} therefore implies that
$D_C^{-1/2}CD_C^{-1/2}$ has a simple zero eigenvalue
and all its other eigenvalues in $[1,2]$, exactly as required.
The nonterminal conditional laws in the sequential sampler remain
positive external-field laws, by the proof of
Lemma~4 of \citet{anarihaqima2026}.
Theorem~1 and Corollary~2 of \citet{anarihaqima2026} now apply
without their conjectured hypothesis and give the marginal and
stretch guarantees, including boundary inputs.
For $k=0$, both outputs are empty; the exact case $k=n$
is deterministic, while the at-most case is included in the cited
Corollary~2.

For $0<k<n$, the exact sampler on the relative interior has law
$\mu_x$ by \citet[Lemma~4]{anarihaqima2026}. Its joint Borel
measurability and the radial continuity of $\mu_x$ in total
variation are proved in Appendix~D of the first version of this
paper \citep{cesaricolomboni2026}.
The boundary-completion argument in \citet[Appendix~E.1]{cesaricolomboni2026}
preserves any finite interior stretch constant: summability along
the radial sequence gives an almost-sure eventual value for each
fixed input, and bounded convergence preserves the law and the
stretch bound for each fixed pair of inputs.
Applying that argument with constant $6$ gives a jointly Borel
exact scheme on the closed hypersimplex with law $\mu_x$ and
stretch at most $6$.

For the additional cardinality assertion, choose the at-most scheme
using the monotone-prefix lift in the proof of
\citet[Corollary~2, Section~7]{anarihaqima2026}: append $k$ dummy
coordinates with total marginal $k-\|x\|_1$, all in $\{0,1\}$
except possibly one, apply the jointly Borel exact scheme just
constructed, and discard the dummies. The lift is continuous and
has $\ell_1$-Lipschitz constant at most $2$, so this gives a
jointly Borel scheme with stretch at most $12$.
For each fixed $x$, the dummy coordinates of marginal zero or one
are respectively excluded or included almost surely, so the remaining
cardinality belongs to
$\{\lfloor\|x\|_1\rfloor,\lceil\|x\|_1\rceil\}$ almost surely.
As in the proof of Corollary~1.5 of the first version of this paper
\citep{cesaricolomboni2026}, whenever this cardinality condition fails,
replace the output by $\{1,\ldots,\lfloor\|x\|_1\rfloor\}$,
with the empty-set convention when $\lfloor\|x\|_1\rfloor=0$.
The failure set and the replacement map are Borel, so this
correction preserves joint Borel measurability.
This enforces the condition for every input and every seed; for each
fixed input, and simultaneously for each fixed pair of inputs,
the correction occurs only on a null event, so all marginals and
the stretch bound remain unchanged.
\end{proof}

\begingroup
\interlinepenalty=10000
\bibliographystyle{plainnat}
\bibliography{biblio}
\endgroup
\end{document}